\documentclass[a4paper]{article}
\usepackage{sectsty}
\usepackage{amsmath}
\usepackage{amsthm}
\usepackage{dsfont}
\usepackage{pxfonts}
\usepackage{etoolbox}

\newtheorem{teo}{Theorem}
\newtheorem{prop}{Proposition}
\newtheorem{coro}{Corollary}

\newtheorem{note}{Remark}

\makeatletter
\def\@seccntformat#1{\@ifundefined{#1@cntformat}%
{\csname the#1\endcsname\quad}
{\csname #1@cntformat\endcsname}
}
\def\section@cntformat{\thesection. \quad}
\def\subsection@cntformat{\thesubsection . \quad}
\makeatother

\begin{document}

\def \ni {\noindent}
\allsectionsfont{\mdseries\itshape\centering}

\def\restriction#1#2{\mathchoice
              {\setbox1\hbox{${\displaystyle #1}_{\scriptstyle #2}$}
              \restrictionaux{#1}{#2}}
              {\setbox1\hbox{${\textstyle #1}_{\scriptstyle #2}$}
              \restrictionaux{#1}{#2}}
              {\setbox1\hbox{${\scriptstyle #1}_{\scriptscriptstyle #2}$}
              \restrictionaux{#1}{#2}}
              {\setbox1\hbox{${\scriptscriptstyle #1}_{\scriptscriptstyle #2}$}
              \restrictionaux{#1}{#2}}}
\def\restrictionaux#1#2{{#1\,\smash{\vrule height .8\ht1 depth .85\dp1}}_{\,#2}} 

\title{Weber's formula for the bitangents of a smooth plane quartic}
\author{Alessio Fiorentino\thanks{This work was supported by the Centre Henri Lebesgue (Programme PIA- ANR-11-LABX-0020-01)}}
\date{}
\maketitle

\begin{abstract}
\ni In a section of his 1876 treatise {\it Theorie der Abel’schen Functionen vom Geschlecht $3$} Weber proved a formula that expresses the bitangents of a non-singular plane quartic in terms of Riemann theta constants ({\it Thetanullwerte}). The present note is devoted to a modern presentation of Weber's formula. In the end a connection with the universal bitangent matrix is also displayed.
\end{abstract}

\medskip

\medskip

\section{Introduction}
The problem of characterizing the complex principally polarized abelian varieties of dimension $g$ which are Jacobian varieties of smooth projective curves of genus $g$, is a long-standing research subject that dates back to Riemann and Schottky.
In fact, the question can be simply answered whenever $g\leq3$, as in this case every indecomposable principally polarized abelian variety is known to be the Jacobian variety of an irreducible smooth projective curve (uniquely determined up to isomorphisms). A naturally related question is how to explicitly recover the curve from a given principally polarized abelian variety. A solution to this problem for non-hyperelliptic curves of genus $3$ has been found in \cite{Guardia2}, where an equation for the curve is recovered by resorting to the Riemann model of the curve associated with Steiner complexes of bitangents. One of the first mathematicians who succeeded in establishing a link between the geometry of the curve and the algebraic structure defined by its period matrix was Heinrich Martin Weber. In his work \cite{Weber} he actually provided both a formula to recover the bitangents of the the curve from its period matrix and a reverse formula to recover the fourth powers of theta constants (\emph{Thetanullwerte}), valued at the point corresponding to the period matrix of the curve, from its bitangents. A detailed explanation of the latter formula, along with a modern proof of it, can be found in \cite{NR}; instead, this note aims to prove Weber's formula for the bitangents of the curve from a modern point of view.

\section{Acknowledgments}
The author wishes to thank Christophe Ritzenthaler for bringing Weber's work to his attention as well as for all the enlightening discussions. He is also grateful to Riccardo Salvati Manni for the several conversations on the subject.

\section{Quadratic forms on symplectic vector spaces over $\mathds{F}_2$}

This brief section is devoted to outlining some of the basic elements concerning the theory of quadratic forms over the finite field $\mathds{F}_2$ and is motivated by the need for a coordinate-free presentation of Weber's formula; a more detailed explanation of the subject can be found in Dolgachev's book \cite{Dolgachev} and in Gross and Harris's paper \cite{GH}.\\
\ni Let $g \geq 1$ be an integer and $V$ a vector space of dimension $2g$ over $\mathds{F}_2$ provided with a symplectic form $\omega$. A \emph{quadratic form} $q$ on the symplectic vector space $(V, \omega)$ is a map $q: V \mapsto \mathds{F}_2$ such that:
\begin{equation*}
\begin{array}{l}
q(\lambda v) = \lambda^2 q(v) \quad \quad \forall \lambda \in \mathds{F}_2 , \, \forall v \in V \\
q(v+w) = q(v) + q(w) +  \omega (v,w) \quad \quad \forall v, w \in V;
\end{array}
\end{equation*}
There are $2^{2g}$ distinct quadratic forms on $(V, \omega)$. By definition $q, q` \in Q(V)$ if and only if $q`-q=\alpha^2$ where $\alpha$ is a linear form on $V$ \footnote{Such a linear form is well defined, as any element of $\mathds{F}_2$ has exactly one square root, which actually coincides with the element itself.}; therefore, for any $q, q` \in Q(V)$ there exists a unique $v \in V$ such that:
\begin{equation}
\label{sumvector}
q`(w) = q(w) +  {\omega(v,w)} \quad \quad \forall w \in V;
\end{equation}
Thanks to (\ref{sumvector}), a free and transitive action of $V$ on $Q(V)$ is well defined by setting $v + q := q'$, hence the set $Q(V)$ is an affine space over $V$, which means it can be identified with $V$ whenever a quadratic form is fixed as origin. Furthermore, the disjoint union $V  \dot{\cup} Q(V)$ can be thought as a vector space of dimension $2g + 1$ over $\mathds{F}_2$ . Once a symplectic basis $e_1, \dots , e_g, f_1, \dots , f_g$ is chosen for $V$, a quadratic form $q_0$ is naturally defined as origin for the affine space $Q(V)$:
\begin{equation}
\label{q0}
q_0(w): = \lambda \cdot \mu = \lambda_1 \mu_1 + \cdots  \lambda_g \mu_g  \quad \quad \forall \, w = (\lambda, \mu)= \sum_{i=1}^g\lambda_ie_i + \sum_{i=1}^g\mu_if_i 
\end{equation}
Then, by (\ref{sumvector}), each $q \in Q(V)$ can be identified with the unique column vector $v=[{}^{m'}_{m''}]$ such that:
\begin{equation}
\label{qintermsofq0}
q(w) = \lambda \cdot \mu + \lambda \cdot m' + m'' \cdot \mu \quad \quad \forall \, w = (\lambda, \mu)
\end{equation}
\ni Besides, since the subgroup of $GL(V)$ that preserves the symplectic form $\omega$ is isomorphic to $SP(2g, \mathds{F}_2)$, an action of $SP(2g, \mathds{F}_2)$ on the affine space $Q(V)$ is well defined by setting $\gamma \cdot q(v) := q({\gamma}^{-1} v)$ for any $v \in V$. The orbits of $Q(V)$ under this action are described in terms of the \emph{Arf invariant} of a quadratic form: 
\begin{equation*}
a(q) := \sum_{i=1}^g q(e_i)q(f_i) \quad \quad \forall q \in Q(V);
\end{equation*}
which does not depend on the choice of the symplectic basis. Then, $Q(V)$ is seen to decompose into two orbits: the set $Q(V)_+$ of \emph{even} quadratic forms, namely those whose Arf invariant is equal to $0$ (the cardinality of this orbit is equal to $2^{g-1}(2^g + 1)$) and the set $Q(V)_-$ of \emph{odd} quadratic forms, namely those whose Arf invariant is equal to $1$ (the cardinality of this orbit is equal to $2^{g-1}(2^g - 1))$. A straightforward computation shows that the quadratic form $q_0$ defined in (\ref{q0}) is even and that $a(q) = m' \cdot m''$ for any quadratic form $q$ whose coordinates with respect to $q_0$ are $[{}^{m'}_{m''}]$, as in (\ref{qintermsofq0}).
\ni Remarkable orbits of non-ordered collections of quadratic forms are also characterized in terms of the Arf invariant; in particular, a triple $q_1, q_2, q_3 \in Q(V)$ is called \emph{syzygetic} (resp. \emph{azygetic}) if $a(q_1) + a(q_2) + a(q_3)+ a(q_1+q_2+q_3) = 0$ (resp. $=1$). Likewise, a collection of quadratic forms $q_1, \dots , q_n \in Q(V)$ with $n \geq 4$ is called syzygetic (resp. azygetic) if each sub-triple $\{q_i, q_j, q_k\} \subset \{q_1,  \dots , q_n\}$  is syzygetic (resp. azygetic). An \emph{Aronhold system} is a collection of $2g + 1$ quadratic forms $q_1,  \dots , q_{2g+1} \in Q(V)$ which is a basis for the vector space $V \dot{\cup} Q(V)$ and such that for any $q =\sum_{i=1}^{2g+1}\lambda_iq_i \in Q(V)$ the following expression holds:
\begin{equation*}
a(q) = \frac12 \left ( \sum_{i=1}^{2g+1}\lambda_i - 1 \right ) + \left \{ \begin{array}{lll} 0 & \text{if} & g \equiv 0, 1 \, \text{mod} \, 4 \\ 1 & \text{if} & g \equiv 2, 3 \, \text{mod} \, 4 \end{array} \right .
\end{equation*} Aronhold systems exist and the action of $SP(2g, \mathds{F}_2)$ on them is transitive.

\section{Theta characteristics and quadratic forms}
This section is intended to recall the link between the above mentioned algebraic settings and the geometry of the projective curves. Classical references for this subject are \cite{ACGH} and \cite{GHprinciples}. We will also follow the exposition outlined in \cite{Guardia} and \cite{NR}.\\
Let $C$ be a smooth complex non-hyperelliptic curve of genus $g$ canonically embedded in $\mathds{P}^{g-1}$ by means of a basis $\omega_1, \dots , \omega_g$ of the cohomology space $H^0(C, \Omega^1)$, and let $\mathfrak{S}_g$ denote the Siegel upper-half space of degree $g$, namely the tube domain of complex symmetric $g \times g$ matrices with positive definite imaginary part. Once a symplectic basis $\delta_1, \dots , \delta_g, \delta'_1, \dots, \delta'_g$ of the homology space $H_1(C, \mathds{Z})$ is chosen, the $g \times 2g$ period matrix of the curve $(\int_{\delta_j}\omega_i, \int_{\delta'_j}\omega_i)$ defines a lattice in $\mathds{C}^g$ and consequently a complex torus whose isomorphism class has a representative of the form $J_C:= \mathds{C}^g/(\mathds{Z}^g + \tau\mathds{Z}^g)$ with $\tau \in \mathfrak{S}_g$. This complex torus is known as the Jacobian variety of the curve $C$ and is a principally polarized abelian variety, whose set of  $2$-torsion points $J_C[2]$ can be clearly identified with the set of the representatives $\frac12 (h + \tau \cdot k)$ with $h,k \in \mathds{Z}_2^g$. Hence, $J_C[2]$ admits a vector space structure over $\mathds{Z}_2$ and is furthermore endowed with a symplectic form given by the Weyl pairing.
The affine space $Q(J_C[2])$ of the quadratic forms on $J_C[2]$ can be therefore identified with $\mathds{Z}_2^{g} \times \mathds{Z}_2^{g}$, once a quadratic form is fixed as origin. This affine space is strictly related to the geometry of the curve, since its points can be identified with the so-called theta characteristics. A \emph{theta characteristic} on $C$ is a divisor $D$ such that $2D \sim K_C$ where $K_C$ is the canonical divisor of the curve. Once a point $P_0 \in C$ is fixed, the well known Abel-Jacobi map $\phi_{P_0}: Div(C) \rightarrow J_C$ is defined on the group $Div(C)$ of the equivalence classes of divisors; since $p=\phi_{P_0}(D'-D) \in J_C[2]$ whenever the divisors $D$ and $D'$ are linearly equivalent to distinct theta characteristics, a free transitive action of the vector space $J_C[2]$ is well defined on the set of theta characteristics as well, by setting $D+p:=D'$. Then, for any theta characteristic $D$ a quadratic form in $Q(J_C[2])$ is uniquely defined by setting:
\begin{equation*}
q_D(p):= [\text{dim}\mathfrak{L}(D+p) + \text{dim}\mathfrak{L}(D)] \, \text{mod} 2
\end{equation*}
where $\mathfrak{L}(D)$ and $\mathfrak{L}(D+p)$ are the Riemann-Roch spaces respectively associated with the divisors $D$ and $D+p$. This gives a bijection between the set of theta characteristics and $Q(J_C[2])$. The theta characteristics can be therefore identified with the vectors in $\mathds{Z}_2^{g} \times \mathds{Z}_2^{g}$ as long as a theta characteristic $D_0$ is fixed; a canonical choice for such a $D_0$ is suggested by Riemann's theorem on the geometry of the theta divisor, as we will briefly recall.\\
\\
\ni A {\it Riemann theta function} of level $2$ with characteristic $m=({}^{m'}_{m''})$, where $m', m'' \in \mathds{Z}^g$ is a holomorphic function  $\theta_{m}: \mathfrak{S}_g \times \mathds{C}^g \mapsto \mathds{C}$ defined by the series:
\begin{equation*}
\theta_{m}(\tau , z) \coloneqq \sum_{n \in \mathds{Z}^g} {\bf e}\left [ {}^t \left ( n + \frac{m'}{2} \right ) \cdot \tau \cdot \left ( n + \frac{m'}{2} \right ) + 2 \left ( n + \frac{m'}{2} \right ) \cdot \left ( z + \frac{m''}{2} \right ) \right ]
\end{equation*} where ${\bf e}(z):= \exp(\pi {\bf i}z)$ and the symbol $\cdot$ stands for the usual inner product.  As a consequence of the {\it reduction formula}:
\begin{equation}
\label{reduction}
\theta_{m+2n}(\tau , z) = (-1)^{m' \cdot n''} \theta_{m} (\tau , z ) \quad \quad \forall m=({}^{m'}_{m''}),  \quad \forall n=({}^{n'}_{n''})
\end{equation} these functions are uniquely determined up to a sign by the so-called reduced characteristics $[m]:=[{}^{m'}_{m''}]$ with $m', m'' \in \mathds{Z}_2^{g}$. The {\it theta constant (Thetanullwert)} with characteristic $m$ is the function defined by setting $\theta_m(\tau):=\theta_m(\tau, 0)$. Riemann theta functions with characteristics satisfy the classical {\it addition formula} (cf. \cite{Igusabook} for a general formulation in terms of real characteristics):
\begin{equation}
\label{addition}
\begin{split}
\theta_{m_1}(\tau, u+v)\theta_{m_2}(\tau, u-v)\theta_{m_3}(\tau)\theta_{m_4}(\tau)=\\
 =1/2^g\sum_{[a] \in  \mathds{Z}^{2g}/ 2\mathds{Z}^{2g}}{\bf e}(m'_1 \cdot a'') \theta_{n_1+a}(\tau, u) \theta_{n_2+a}(\tau, u)\theta_{n_3+a}(\tau, v)\theta_{n_4+a}(\tau, v) \\
\end{split}
\end{equation}
where the sum runs over a set of representatives for $\mathds{Z}^{2g}/ 2\mathds{Z}^{2g}$ and $\{ n_1, n_2, n_3, n_4 \}$ and $\{ m_1, m_2, m_3, m_4 \}$ are any two collections of four characteristics that satisfy the following identity:
\begin{equation*}
(n_1, n_2, n_3, n_4)=\frac{1}{2}(m_1, m_2, m_3, m_4) \cdot \begin{pmatrix} 1 & 1 & 1 & 1 \\ 1 & 1 & -1 & -1 \\ 1 & -1 & 1 & -1 \\ 1 & -1 & -1 & 1 \end{pmatrix}
\end{equation*}

\ni If $\tau \in \mathfrak{S}_g $ identifies the Jacobian variety $J_C$ of the curve, the \emph{theta divisor} $\Theta$ is well defined on $J_C$ as the pull-back of the divisor $\{z \in \mathds{C}^g \mid \theta_0(\tau) = 0 \}$, and the Chern class of the holomorphic line bundle associated with such a divisor gives a principal polarization on $J_C$. The following classical theorem holds:

\begin{teo}[\bf{Riemann's theorem}]
There exists a theta characteristic $D_0$ on the curve $C$ such that:
\begin{equation*}
W_{g-1} = \Theta + D_0
\end{equation*}
where $W_{g-1}:= \{D \in Div(C) \, \mid \text{deg}(D)=g-1 \, , \text{dim}\mathfrak{L}(D) >0 \}$. Furthermore,  $\text{dim}\mathfrak{L}(D)$ is even, and $\text{mult}_p(\Theta) = \text{dim}\mathfrak{L}(D+p)$ for any $p \in J_C[2]$

\end{teo}

\ni If such a theta characteristic $D_0$ is fixed, a quadratic form $q_0$ in $Q(J_C[2])$ is fixed as well; then, any theta characteristic on the curve is of the form $D_0+v$ with $v =[{}^{m'}_{m''}]$ and $m', m'' \in \mathds{Z}_2^{g}$ and the corresponding quadratic form is $q_0 + v$. A Riemann theta function $\theta_{[m]}$ with reduced characteristic $[m]=[{}^{m'}_{m''}]$ can be therefore regarded as a function $\theta[q]$ associated with the quadratic form $q= q_0 + [{}^{m'}_{m''}]$. The function $z \rightarrow \theta[q](\tau, z)$ is even (resp. odd) whenever $q$ is even (resp. odd), hence the theta constant $\theta[q]$ is non-trivial if and only if $q$ is even; furthermore, for any $q= q_0 + [{}^{m'}_{m''}]$ and for any $(k,h) \in \mathds{Z}_2^{g} \times \mathds{Z}_2^{g}$ the following transformation law holds (cf. \cite{FR}):
\begin{equation*}
\label{transformation}
\theta[q]\left (\tau, z+\frac12 h + \frac12 \tau \cdot k \right ) = {\bf e}\left ( - \frac 12 k \cdot (m'' + h) - k \cdot z -\frac 14 {}^t k \cdot \tau \cdot k   \right ) \theta[q + [{}^{k}_{h}]](\tau, z ) 
\end{equation*}
\ni Thanks to this formula the pull-back of the zero locus of any Riemann theta function $\theta[q]$ also defines a divisor $\Theta[q]$ on $J_C$ and $\text{mult}_0(\Theta[q+v])=\text{mult}_v(\Theta)$ for any $v \in J_C[2]$. Riemann's theorem thus implies that the effective theta divisors are those associated with odd quadratic forms. Therefore, for any odd quadratic form $q \in Q(J_C[2])_-$ the associated theta characteristic is of the type $D_q=P_1 + \cdots + P_{g-1}$  and is actually the divisor that is cut on the canonical curve by a hyperplane tangent at the image points of $P_1, \dots, P_{g-1}$ in $\mathds{P}^{g-1}$ under the canonical map; the direction of such a hyperplane in $\mathds{P}^{g-1}$ is then given by the gradient of the corresponding Riemann theta function valued at $z=0$:
\begin{equation*}
{\rm grad}^0_z\theta[q](\tau):= \left ( \frac {\partial \theta[q]}{\partial z_1}(\tau, 0), \dots ,  \frac {\partial \theta[q]}{\partial z_g}(\tau, 0) \right )
\end{equation*}
\ni which is non-trivial if and only if $q$ is odd. The Jacobian determinant of $g$ Riemann theta functions valued at $z=0$ will be henceforward denoted by:
\begin{equation}
\label{jacobians}
D[q_1, \dots, q_g] (\tau) := ({\rm grad}^0_z\theta [q_1] \wedge \dots \wedge {\rm grad}^0_z\theta [q_g]) (\tau)
\end{equation}

\ni The algebraic link between theta constants and Jacobian determinants is displayed by Igusa's conjectural formula (cf. \cite{Igusa}), which has been proved up to the case $g=5$ (cf. \cite{Frobenius} and \cite{Fay}). In the next section we shall resort to a coordinate-free version of the formula for ratios of determinants with explicit signs, which can be derived from the addition formula. 

\section{Bitangents of a plane quartic: Weber's formula}

For the rest of the paper we will be only concerned with the $g=3$ case.
The canonical model of a non singular curve $C$ of genus $3$ is a smooth plane quartic, whose 28 bitangents are in bijection with the $28$ odd quadratic forms on $J_C[2]$. By virtue of the geometrical link recalled in the previous section, there exist homogeneus coordinates $(Z_1: Z_2:Z_3)$  in $\mathds{P}^2$ such that the equations of the $28$ bitangents are:
\begin{equation} 
\label{bitgrad}
\sum_{i=1}^3 \frac{\partial \theta [q]}{\partial Z_i}(\tau,0) Z_i=0, \qquad  \forall q \in Q(J_C[2])_-
\end{equation}

\ni In this case, an Aronhold system is a collection of seven odd quadratic forms $q_1, \dots , q_7$ such that each sub-triple $\{q_i, q_j, q_k\} \subset \{q_1,  \dots , q_7\}$ is azygetic, which means $q_i + q_j + q_k$ is even; there exist exactly $288$ distinct Aronhold systems when $g=3$. Once an Aronhold system is fixed, the remaining $21$ odd quadratic forms can be simply described in terms of it as follows:
\begin{equation}
\label{21qforms}
q_{ij}:= q_S + q_i + q_j \quad \quad \forall i \neq j 
\end{equation}  where $q_S:= \sum_{i=1}^7q_i$ is an even quadratic form. The other $35$ even quadratic forms different form $q_S$ are easily seen to be described in terms of the Aronhold system as follows:
\begin{equation}
\label{evenforms}
q_{ijk}:= q_i + q_j + q_k \quad \quad \forall i, j , k \quad \text{distinct} 
\end{equation}

\ni An Aronhold system of bitangents for the plane quartic is then a collection of seven bitangents associated with an Aronhold system of quadratic forms; this geometrically translates into the condition that for any collection of three bitangents out of the seven, the six corresponding points of tangency on the quartic do not lie in the same conic. The datum of an Aronhold system is enough to recover an equation for the plane quartic along with equations for the remaining $21$ bitangents; this is basically done by means of the Steiner complexes of bitangents determined by the sub-collections of six bitangents in the Aronhold system. We will only recall here the main features of the method of reconstruction with a particular focus on the Riemann model of the curve (cf. \cite{Weber} for details and \cite{Dolgachev} for a modern exposition of the subject).\\

\ni The following statement holds:

\begin{prop}
\label{model}
Let $q$ be a non-null quadratic form on $J_C[2]$ and let $\{q_1, q'_1\}$, $\{q_2, q'_2\}$ and $\{q_3, q'_3\}$ be three pairs of odd quadratic forms on $J_C[2]$ such that $q_i+q'_i=q$ for any $i=1,2,3$. Then, for any two of these pairs there exists a conic that passes through the eight points of tangency; in particular, an equation for the quartic is given by:
\begin{equation}
\label{modelequation}
4 f_1\xi_1f_2\xi_2-(f_1 \xi_1 + f_2 \xi_2 + f_3 \xi_3 )^2=0
\end{equation}
or, in Weber's notation:
\begin{equation*}
\sqrt{f_1\xi_1} + \sqrt{f_2\xi_2}+\sqrt{f_3\xi_3}=0
\end{equation*}
where $\{f_i, \xi_i\}$ is a suitable pair of linear forms associated with the bitangents corresponding to the pair $\{q_i, q'_i\}$.
\end{prop}

\ni As any subtriple $q_i, q_j, q_k$ of an Arnohold system is an azygetic triple, it can be completed to three pairs $\{q_i, q'_i\}$, $\{q_j, q'_j\}$ and $\{q_k, q'_k\}$ such as in the statement of Proposition \ref{model}. Thus, any three bitangents in an Aronhold system cannot intersect at a same point, because such a point would be a singular point of the curve by (\ref{modelequation}), while the curve is smooth; this proves the following:

\begin{coro}
Up to a projective transformation, an Aronhold system of bitangents for the quartic is given by the following equations in $\mathds{P}^2$:\\
\begin{equation}
\label{eqbit}
\begin{array}{lll}
\beta_1: \quad X_1=0  & & \beta_5 : \quad a_{11}X_1 + a_{12}X_2 + a_{13}X_3 = 0 \\
\beta_2: \quad X_2=0 & & \beta_6 : \quad a_{21}X_1 + a_{22}X_2 + a_{23}X_3 = 0 \\
\beta_3: \quad X_3=0 & & \beta_7 : \quad a_{31}X_1 + a_{32}X_2 + a_{33}X_3 = 0 \\
\beta_4: \quad X_1+X_2+X_3=0
\end{array}
\end{equation} for suitable $(a_{i1}: a_{i2}:a_{i3}) \in \mathds{P}^2$.
\end{coro}

\ni A proof of the following classical result will be omitted here, as it can be found in \cite{Weber}:
\begin{prop}[\bf{Riemann's model}]
\label{Rmodel}
Let $\beta_1, \dots , \beta_7$ an Aronhold system of bitangents for the curve as in (\ref{eqbit}) and $q_1, \dots , q_7$ the corresponding quadratic forms. The three pairs $\{q_1, q_{23}\}$, $\{q_2, q_{13}\}$ and $\{q_3, q_{12}\}$ (cf. (\ref{21qforms})) are such as in the statement of Proposition \ref{model}, and an equation for the curve is given by:

\begin{equation*}
4 X_1\xi_{23}X_2\xi_{13}=(X_1 \xi_{23} + X_2 \xi_{13} + X_3 \xi_{12} )^2
\end{equation*}
where $\xi_{ij}$ are linear forms associated with the bitangents corresponding to $q_{ij}$ and determined by the linear system:
\begin{equation}
\label{Riemannmodel}
\left \{
\begin{array}{l}
\xi_{23} + \xi_{13} + \xi_{12}+X_1 + X_2 + X_3 = 0;
\\
\frac{\xi_{23}}{a_{i1}} + \frac{\xi_{13}}{a_{i2}} + \frac{\xi_{12}}{a_{i3}}+k_i(a_{i1}X_1 +a_{i2} X_2 + a_{i3}X_3)= 0 \quad \quad i=1,2,3
\end{array}
\right.
\end{equation} with $k_1, k_2, k_3 \in \mathds{C}^*$ unique solution of the linear system:
\begin{equation}
\label{k}
\begin{pmatrix} \lambda_1 {a_{11}} & \lambda_2{a_{21}} & \lambda_3{a_{31}} \\ \lambda_1{a_{12}} & \lambda_2{a_{22}} & \lambda_3{a_{32}} \\ \lambda_1{a_{13}} & \lambda_2{a_{23}} & \lambda_3{a_{33}} \end{pmatrix} \begin{pmatrix} k_1 \\ k_2 \\ k_3 \end{pmatrix} = \begin{pmatrix} -1 \\ -1 \\ -1 \end{pmatrix} 
\end{equation}
\ni where $\lambda_1, \lambda_2, \lambda_3 \in \mathds{C}^*$ are such that:
\begin{equation*}
 {\begin{pmatrix} \frac{1}{a_{11}} & \frac{1}{a_{21}} & \frac{1}{a_{31}} \\ \frac{1}{a_{12}} & \frac{1}{a_{22}} & \frac{1}{a_{32}} \\ \frac{1}{a_{13}} & \frac{1}{a_{23}} & \frac{1}{a_{33}} \end{pmatrix}} \begin{pmatrix} \lambda_1 \\ \lambda_2 \\ \lambda_3 \end{pmatrix} = \begin{pmatrix} -1 \\ -1 \\ -1 \end{pmatrix}
\end{equation*}

\end{prop}

\ni Note that the unicity of such a construction is a consequence of the results proved by Caporaso and Sernesi (cf. \cite{CapSer}) and by Lehavi (cf. \cite{Lehavi}).\\

\ni As the curve $C$ is fixed, for the sake of simplicity we shall omit the symbol of the variable $\tau$ in the expressions of theta functions and theta constants throughout the rest of this section. Furthermore, by a slight abuse of notation we shall denote by $(q)=(q',q'')$ the non-reduced characteristic that corresponds to the coordinates of the quadratic form $q$ with respect to a fixed quadratic form $q_0$ and by $(\sum_i q_i)$ the non-reduced characteristic $\sum_i (q_i)$. To prove Weber's formula we need the following Proposition first. 

\begin{prop}
\label{jacobi}
Let $\{q_1, q_2, q_3, q_4\}$ any azygetic $4$-tuple of odd quadratic forms, and let $\{q_5, q_6, q_7\}$ one of the two distinct triples which complete the $4$-tuple to an Aronhold system $\{q_1, q_2, q_3, q_4, q_5, q_6, q_7\}$. Then:
\begin{equation*}
\frac{D[q_4, q_2, q_3]}{D[q_1, q_2, q_3]}= - {\bf e}((q_5+q_6+q_7)' \cdot (q_1+q_4)'') \frac{\theta(q_5+q_6+q_1)\theta(q_5+q_7+q_1)\theta(q_6+q_7+q_1)}{\theta(q_5+q_6+q_4)\theta(q_5+q_7+q_4)\theta(q_6+q_7+q_4)} 
\end{equation*}
\ni where $D[q_i,q_j, q_k]$ are the Jacobian determinants of the corresponding Riemann theta functions with reduced characteristics valued at $z=0$, as in (\ref{jacobians}). 
\end{prop}

\begin{proof}
If we set $u=0$ in the formula (\ref{addition}) and choose $n_1=(q_5+q_6)$, $n_2=(q_5+q_7)$, $n_3=(q_6+q_7)$ and $n_4=0$, we get for any $z \in \mathds{C}^g$:
\begin{equation*}
0= \sum_{q \in  Q(J_C[2])}\chi(q)\theta(q_5+q_6+q)\theta(q_5+q_7+q)\theta(q_6+q_7+q)( z)\theta(q)(z)
\end{equation*} where $\chi(q):={\bf e}((q_5+q_6+q_7)' \cdot q'')$. The right side of the identity is the sum of two terms $S_-$ and $S_+$, obtained by letting $q$ run respectively over $Q(J_C[2])_-$ and over $Q(J_C[2])_+$. Thanks to the labelling  introduced in (\ref{21qforms}) for the elements of $Q(J_C[2])_-$ one easily derives $S_-=S_-^{(4)}+S_-^{(6)}$, where:
\begin{equation*}
S_-^{(4)}=\sum_{i=1}^4\chi(q_i) \, \theta(q_5+q_6+q_i)\theta(q_5+q_7+q_i)\theta(q_6+q_7+q_i)( z)\theta[q_i](z)
\end{equation*}
\begin{equation*}
S_-^{(6)}= \sum_{j,k \in \{1,2,3,4\} \, \text{s.t.} \, j < k }\chi(q_{jk}) \theta(q_5+q_6+q_{jk})\theta(q_5+q_7+q_{jk})\theta(q_6+q_7+q_{jk})( z)\theta[q_{jk}](z) \\
\end{equation*}
As for $S_+$, the labelling introduced in (\ref{evenforms}) for the elements of $Q(J_C[2])_+$ shows that $S_+=S_+^{(4)}+S_+^{(6)}$ where $S_+^{(4)}$ is the term given by summing on the four quadratic forms $q_{i67}$ with $i \in \{1,2,3,4\}$, while $S_+^{(6)}$ is the term given by summing on the six quadratic forms $q_{5jk}$ with $j,k \in \{1,2,3,4\}$ and $j <k$. A straightforward computation with the reduction formula shows that $S_+^{(6)}$ and $S_-^{(6)}$ cancel out, whereas:
\begin{equation*}
S_+^{(4)}=\sum_{i=1}^4{\bf e}(a(q_6)+a(q_7)) \, \chi(q_i) \, \theta(q_5+q_6+q_i)\theta(q_5+q_7+q_i)\theta(q_6+q_7+q_i)( z)\theta[q_i](z) = S_-^{(4)}
\end{equation*}
\ni Therefore, one finally obtains the identity:
\begin{equation*}
\sum_{k=1}^{4} \chi(q_k) \theta(q_5+q_6+q_k) \theta(q_5+q_7+q_k)\theta(q_6+q_7+q_k)(z) \theta[q_k](z)=0 \quad \quad \forall z \in \mathds{C}^g
\end{equation*}
By taking the derivative with respect to each $z_j$ for $j=1,2,3$ and evaluating the resulting expression at $z=0$ one obtains the following identities:
\begin{equation*}
\sum_{k=1}^{4} \chi(q_k) \theta(q_5+q_6+q_k) \theta(q_5+q_7+q_k)\theta(q_6+q_7+q_k) \restriction{\frac{\partial\theta[q_k]}{\partial z_j}}{\, z=0}=0 \quad \quad j=1,2,3
\end{equation*} from which the statement clearly follows.
\end{proof}

\ni We can now state the main theorem of this note:

\begin{teo}[\bf{Weber's formula}]
Let $\tau \in \mathfrak{S}_3$ the period matrix of a smooth plane quartic $C$. If $q_1, \cdots q_7$ is an Aronhold system of quadratic forms on the $2$-torsion points of the Jacobian variety $\mathds{C}^3/(\mathds{Z}^3 + \tau\mathds{Z}^3)$, then for the coefficients in (\ref{eqbit}) one has:
\begin{equation*}
a_{ij} = \eta_i \, {\bf e}(q'_j \cdot (q_4+q_{4+i})'') \frac{\theta(q_4+q_r+q_j)\theta(q_4+q_s+q_j)}{\theta(q_{4+i}+q_r+q_j)\theta(q_{4+i}+q_s+q_j)} \quad \quad i,j=1,2,3
\end{equation*}
where $r$ and $s$ are such that $\{4+i, r, s\} = \{5, 6, 7 \}$ and $\eta_i$ is a non-zero scalar factor that only depends on the index $i$, which is due to the fact that the equations for $\beta_5$, $\beta_6$ and $\beta_7$  in (\ref{eqbit}) are defined up to a scalar.
\end{teo}

\begin{note}
\label{reducedformula}
\ni The reduction formula (\ref{reduction}) can be used in Weber's formula to express the coefficients of the bitangents in terms of reduced characteristics. In this case one has: 
\begin{equation*}
a_{ij} =  \rho_{ij} \cdot \, \eta_i \, {\bf e}(q'_j \cdot (q_4+q_{4+i})'')\frac{\theta[q_4+q_r+q_j]\theta[q_4+q_s+q_j]}{\theta[q_{4+i}+q_r+q_j]\theta[q_{4+i}+q_s+q_j]} \quad \quad i,j=1,2,3
\end{equation*} where, for any $i$ and $j$, $\rho_{ij}$ is the product of the reduction signs (cf. (\ref{reduction})) of the four theta constants appearing in the expression of $a_{ij}$.
\end{note} 

\begin{proof}[Proof of Weber's formula]
Let $f_i=f_i(X_1, X_2, X_3)$ be linear forms associated with the bitangent $\beta_i$ for any $i=1, \dots 7$; the equations (\ref{eqbit}) yield the following linear system for the $f_i$:
\begin{equation}
\label{equationsf}
\left \{
\begin{array}{l}
f_4 = f_1 + f_2 + f_3;
\\
f_5= a_{11}f_1 + a_{12}f_2 + a_{13}f_3
\\
f_6= a_{21}f_1 + a_{22}f_2 + a_{23}f_3
\\
f_7=a_{31}f_1 + a_{32}f_2 + a_{33}f_3 

\end{array}
\right.
\end{equation}
By (\ref{bitgrad}), there also exists a projective transformation $\varphi: \mathds{P}^2 \mapsto \mathds{P}^2$ such that:
\begin{equation*} 
f_i (X_1, X_2, X_3) = h_i \sum_{j=1}^3 \restriction{\frac{\partial \theta [q_i]}{\partial Z_j}}{\, z=0} \, \varphi_j(X_1, X_2, X_3) \qquad  \forall i = 1, \dots , 7
\end{equation*} with suitable coefficients $h_i \in \mathds{C}^*$. Thus, each equation in (\ref{equationsf}) yelds linear systems in the variables $h_i$ and $a_{ij}$:
\begin{equation}
\label{s4}
h_4 \restriction{\frac{\partial \theta [q_4]}{\partial Z_j}}{\, z=0} = \sum_{i=1}^3h_i \restriction{\frac{\partial \theta [q_i]}{\partial Z_j}}{\, z=0} \quad \quad j = 1,2,3
\end{equation}
\begin{equation}
\label{s567}
h_{4+i} \restriction{\frac{\partial \theta [q_{4+i}]}{\partial Z_j}}{\, z=0} = \sum_{l=1}^3a_{il}h_l \restriction{\frac{\partial \theta [q_l]}{\partial Z_j}}{\, z=0} \quad \quad j = 1,2,3 \quad i=1, 2, 3
\end{equation}
From (\ref{s4}) one has:
\begin{equation*}
h_1=\frac{D[q_4, q_2, q_3]}{D[q_1, q_2, q_3]}h_4; \quad h_2=\frac{D[q_1, q_4, q_3]}{D[q_1, q_2, q_3]}h_4; \quad h_3=\frac{D[q_1, q_2, q_4]}{D[q_1, q_2, q_3]}h_4;
\end{equation*}
\ni By replacing these solutions into (\ref{s567}), one gets the coefficients for $\beta_{4+i}$ for $i=1,2,3$:
\begin{equation}
\label{Aronholddet}
a_{i1}=\mu_i\frac{D[q_{4+i}, q_2, q_3]}{D[q_4, q_2, q_3]}; \, \quad \, a_{i2}=\mu_i \frac{D[q_1, q_{4+i}, q_3]}{D[q_1, q_4, q_3]}; \, \quad \,  a_{i3}=\mu_i\frac{D[q_1, q_2, q_{4+i}]}{D[q_1, q_2, q_4]};
\end{equation} where $\mu_i := h_{4+i}/h_4 \in \mathds{C}^*$. Therefore, the bitangents $\beta_{4+i}$ are uniquely determined as points in $\mathds{P}^2$ by duality. By repeating the same procedure as before with equations (\ref{Riemannmodel}), one obtains for suitable coefficients $h_{23}, h_{13}, h_{12} \in \mathds{C}^*$:
\begin{equation*}
h_4 \restriction{\frac{\partial \theta [q_4]}{\partial Z_j}}{\, z=0} = h_{23} \restriction{\frac{\partial \theta [q_{23}]}{\partial Z_j}}{\, z=0} + h_{13} \restriction{\frac{\partial \theta [q_{13}]}{\partial Z_j}}{(\, z=0} + h_{12} \restriction{\frac{\partial \theta [q_{12}]}{\partial Z_j}}{(\, z=0}
\end{equation*}
\begin{equation*}
k_i h_{4+i} \restriction{\frac{\partial \theta [q_{4+i}]}{\partial Z_j}}{\, z=0} = \frac{h_{23}}{a_{i1}} \restriction{\frac{\partial \theta [q_{23}]}{\partial Z_j}}{\, z=0} + \frac{h_{13}}{a_{i2}} \restriction{\frac{\partial \theta [q_{13}]}{\partial Z_j}}{(\, z=0} + \frac{h_{12}}{a_{i3}} \restriction{\frac{\partial \theta [q_{12}]}{\partial Z_j}}{\, z=0}
\end{equation*} with $i, j = 1,2,3$. By solving these linear systems, one has likewise:
\begin{equation*}
\frac{1}{a_{i1}}=k_i \mu_i\frac{D[q_{4+i}, q_{13}, q_{12}]}{D[q_4, q_{13}, q_{12}]}; \, \quad \,\frac{1}{a_{i2}}=k_i\mu_i\frac{D[q_{23}, q_{4+i}, q_{12}]}{D[q_{23}, q_4, q_{12}]}; \, \quad \, \frac{1}{a_{i3}}=k_i\mu_i \frac{D[q_{23}, q_{13}, q_{4+i}]}{D[q_{23}, q_{13}, q_4]};
\end{equation*} 
\ni Thus, by applying Proposition \ref{jacobi} to the azygetic $4$-tuples of the two Aronhold systems:
\begin{equation*}
\{q_1, q_2, q_3, q_4, q_5, q_6, q_7 \} \quad \quad \{q_{23}, q_{13}, q_{12}, q_4, q_5, q_6, q_7 \}
\end{equation*}
one gets an explicit expression for the square power of the constant factor in terms of theta constants:
\begin{equation*}
\begin{split}
\mu^2_i=&\frac{1}{k_i}\frac{D[q_4, q_2, q_3]D[q_4, q_{13}, q_{12}]}{D[q_{4+i}, q_2, q_3]D[q_{4+i}, q_{13}, q_{12}]}=\\
& =\frac{1}{k_i}{\bf e}(( q_4+q_{4+i})' \cdot (q_4+q_5+q_6+q_7)'')\frac{\theta^2(q_{4+i}+q_r+q_s)}{\theta^2(q_4+q_r+q_s)} \quad i=1,2,3\\
\end{split}
\end{equation*}
where $r$ and $s$ are such that $\{4+i, r, s\} = \{5,6,7\}$. Therefore, by replacing this expression into (\ref{Aronholddet}) one finally has:
\begin{equation*}
a_{ij} = -\epsilon_i e^{\frac{\pi}{2}{\bf i} ( q_4+q_{4+i})' \cdot (q_4+q_5+q_6+q_7)''} {\bf e}( (q_j+q_r+q_s)' \cdot (q_4+q_{4+i})'') \frac{\theta(q_4+q_r+q_j)\theta(q_4+q_s+q_j)}{\theta(q_{4+i}+q_r+q_j)\theta(q_{4+i}+q_s+q_j)}
\end{equation*}
where $\epsilon_i$ is a fixed root of $1/k_i$ for any $i=1,2,3$, and $-{\bf e}((q_r+q_s)' \cdot (q_4+q_{4+i})'')$ is a sign that can be absorbed into the definition of the root, as it only depends on the index $i$. This proves the statement.
\end{proof}

\ni The following corollary follows as a straightforward consequence:
\begin{coro}
\label{choicek1}
By setting in Weber's formula:
\begin{equation*}
\eta_i:=\epsilon_i e^{\frac{\pi}{2}{\bf i}( q_4+q_{4+i})' \cdot (q_4+q_5+q_6+q_7)''} \quad \quad \quad i=1,2,3
\end{equation*}
where $\epsilon_i$ is a chosen sign that only depends on $i$, the corresponding choice of representatives for the points $(a_{i1}: a_{i2}: a_{i3})$ in $\mathds{P}^2$ for $i=1,2,3$ is such that $(k_1, k_2, k_3)=(1,1,1)$ is the unique solution of (\ref{k}). 
\end{coro}

\begin{proof}
By the proof of Weber's formula one has $\eta_i=\sigma_i e^{\frac{\pi}{2}{\bf i} ( q_4+q_{4+i})' \cdot (q_4+q_5+q_6+q_7)''}$ where $\sigma_i$ is a non-zero factor that reduces to a sign for each $i=1,2,3$ if and only if $k_1=k_2=k_3=1$.
\end{proof}

\ni As an example, we can fix a system of coordinates for the quadratic forms as in (\ref{qintermsofq0}) and consider the following Aronhold system in terms of reduced characteristics:

\begin{equation*}
n_1= \begin{bmatrix}1 1 1 \\ 1 1 1 \end{bmatrix}; \, n_2=\begin{bmatrix} 0 0 1 \\ 0 1 1 \end{bmatrix}; n_3= \begin{bmatrix} 0 1 1\\ 0 0 1 \end{bmatrix}; n_4=\begin{bmatrix}1 0 1 \\ 1 0 0 \end{bmatrix}; n_5= \begin{bmatrix} 1 0 0 \\ 1 0 1 \end{bmatrix};n_6=\begin{bmatrix} 1 1 0 \\ 0 1 0 \end{bmatrix}; n_7=\begin{bmatrix} 0 1 0 \\ 1 1 0 \end{bmatrix};
\end{equation*}
\\
\ni Then, we can apply Weber's formula with the choice made in Corollary \ref{choicek1} and compute the reduction signs (see Remark \ref{reducedformula}):
\begin{equation*}
\begin{array}{ccc}
\rho_{11}= +1; & \, \rho_{21}=+1; &\, \rho_{31}= +1; \\
\rho_{12}= +1; & \, \rho_{22}=+1; &\, \rho_{32}= +1; \\
\rho_{13}= +1; & \, \rho_{23}=-1; &\, \rho_{33}= -1;
\end{array}
\end{equation*}
so as to obtain Weber's result (cf. \cite{Weber}):

\begin{equation*}
\begin{array}{ccc}
a_{11}= \epsilon_1 {\bf i} \frac{\theta{\begin{bmatrix}1 0 0 \\ 0 0 1 \end{bmatrix}} \theta{\begin{bmatrix}0 0 0 \\ 1 0 1 \end{bmatrix}}}{\theta{\begin{bmatrix}1 0 1 \\ 0 0 0 \end{bmatrix}} \theta{\begin{bmatrix}0 0 1 \\ 1 0 0 \end{bmatrix}}};  & a_{21}= \epsilon_2{\bf i} \frac{\theta{\begin{bmatrix}1 1 0 \\ 1 1 0 \end{bmatrix}} \theta{\begin{bmatrix}0 0 0 \\ 1 0 1 \end{bmatrix}}}{\theta{\begin{bmatrix}1 0 1 \\ 0 0 0 \end{bmatrix}} \theta{\begin{bmatrix}0 1 1 \\ 0 1 1 \end{bmatrix}}}; & a_{31}= - \epsilon_3 \frac{\theta{\begin{bmatrix}1 1 0 \\ 1 1 0 \end{bmatrix}} \theta{\begin{bmatrix}1 0 0 \\ 0 0 1 \end{bmatrix}}}{\theta{\begin{bmatrix}0 0 1 \\ 1 0 0 \end{bmatrix}} \theta{\begin{bmatrix}0 1 1 \\ 0 1 1 \end{bmatrix}}};\\
\\
a_{12}= \epsilon_1 {\bf i} \frac{\theta{\begin{bmatrix}0 1 0 \\ 1 0 1 \end{bmatrix}} \theta{\begin{bmatrix}1 1 0 \\ 0 0 1 \end{bmatrix}}}{\theta{\begin{bmatrix}0 1 1 \\ 1 0 0 \end{bmatrix}} \theta{\begin{bmatrix}1 1 1 \\ 0 0 0 \end{bmatrix}}};  & a_{22}= \epsilon_2 {\bf i} \frac{\theta{\begin{bmatrix}0 0 0 \\ 0 1 0 \end{bmatrix}} \theta{\begin{bmatrix}1 1 0 \\ 0 0 1 \end{bmatrix}}}{\theta{\begin{bmatrix}0 1 1 \\ 1 0 0 \end{bmatrix}} \theta{\begin{bmatrix}1 0 1 \\ 1 1 1 \end{bmatrix}}}; & a_{32}= \epsilon_3 \frac{\theta{\begin{bmatrix}0 0 0 \\ 0 1 0 \end{bmatrix}} \theta{\begin{bmatrix}0 1 0 \\ 1 0 1 \end{bmatrix}}}{\theta{\begin{bmatrix}1 1 1 \\ 0 0 0 \end{bmatrix}} \theta{\begin{bmatrix}1 0 1 \\ 1 1 1 \end{bmatrix}}};\\
\\
a_{13}= \epsilon_1 {\bf i} \frac{\theta{\begin{bmatrix}0 0 0 \\ 1 1 1 \end{bmatrix}} \theta{\begin{bmatrix}1 0 0 \\ 0 1 1 \end{bmatrix}}}{\theta{\begin{bmatrix}0 0 1 \\ 1 1 0 \end{bmatrix}} \theta{\begin{bmatrix}1 0 1 \\ 0 1 0 \end{bmatrix}}};  & a_{23}= \epsilon_2 {\bf i} \frac{\theta{\begin{bmatrix}0 1 0 \\ 0 0 0 \end{bmatrix}} \theta{\begin{bmatrix}1 0 0 \\ 0 1 1 \end{bmatrix}}}{\theta{\begin{bmatrix}0 0 1 \\ 1 1 0 \end{bmatrix}} \theta{\begin{bmatrix}1 1 1 \\ 1 0 1 \end{bmatrix}}}; & a_{33}= \epsilon_3 \frac{\theta{\begin{bmatrix}0 1 0 \\ 0 0 0 \end{bmatrix}} \theta{\begin{bmatrix}0 0 0 \\ 1 1 1 \end{bmatrix}}}{\theta{\begin{bmatrix}1 0 1 \\ 0 1 0 \end{bmatrix}} \theta{\begin{bmatrix}1 1 1 \\ 1 0 1 \end{bmatrix}}};\\

\end{array}
\end{equation*}

\begin{note}
The formula in (\ref{Aronholddet}) for the Aronhold system (\ref{eqbit}) is in accordance with the modular description of the universal matrix of bitangents obtained in \cite{F}. If a system of coordinates for the quadratic forms is fixed as in (\ref{qintermsofq0}), an Aronhold system $q_1, \dots ,  q_7$ such that  $q_0=\sum_{i=1}^7q_i$ is given by:
\begin{equation*}
n_1= \begin{bmatrix}1 1 1 \\ 1 1 1 \end{bmatrix}; \, n_2=\begin{bmatrix}1 1 0 \\ 1 0 0 \end{bmatrix}; n_3= \begin{bmatrix}1 0 1 \\ 0 0 1 \end{bmatrix}; n_4=\begin{bmatrix}1 0 0 \\ 1 1 0 \end{bmatrix}; n_5= \begin{bmatrix} 0 1 0 \\ 0 1 1 \end{bmatrix};n_6=\begin{bmatrix} 0 0 1 \\ 1 0 1 \end{bmatrix}; n_7=\begin{bmatrix} 0 1 1 \\ 0 1 0 \end{bmatrix};
\end{equation*}
\ni Then the first row of the universal bitangent matrix (cf. \cite{F}) gives the following modular expressions for the corresponding bitangents:
\begin{equation*}
\begin{array}{lll}
\beta'_1: \quad D[n_1+n_4, n_1+n_2, n_1+n_3]\sum_{j=1}^3 \restriction{\frac{\partial \theta_{n_1}}{\partial Z_j}}{, z=0} Z_j=0 & & \\
\beta'_2: \quad D[n_ 2+n_4 , n_1+n_2, n_2+n_3]\sum_{j=1}^3 \restriction{\frac{\partial \theta_{n_2}}{\partial Z_j}}{\, z=0} Z_j=0 & & \\
\beta'_3: \quad D[n_1, n_2, n_4]\sum_{j=1}^3 \restriction{\frac{\partial \theta_{n_3}}{\partial Z_j}}{\, z=0} Z_j=0 & & \\
\beta'_i: \quad D[n_1, n_2, n_3]\sum_{j=1}^3 \restriction{\frac{\partial \theta_{n_i}}{\partial Z_j}}{\, z=0} Z_j=0  & &  i=4, 5, 6, 7 \\
\end{array}
\end{equation*}
\ni where, as above, $D[n_i,n_j, n_k]:= {\rm grad}^0_z\theta_{n_i} \wedge {\rm grad}^0_z\theta_{n_j} \wedge {\rm grad}^0_z\theta_{n_k}$. A straightforward computation shows that the ordered collection of bitangents $\beta_1, \dots \beta_7$ given by Weber's formula is sent to the ordered collection $\beta'_1, \dots \beta'_7$ by the projective transformation $\phi: \mathds{P}^2 \rightarrow \mathds{P}^2$, defined by the matrix:
\begin{equation*}
A_{\phi}:= \begin{pmatrix} D[n_4, n_2, n_3 ]\restriction{\frac{\partial \theta_{n_1}}{\partial Z_1}}{\, z=0} &  D[n_1, n_4, n_3] \restriction{\frac{\partial \theta_{n_2}}{\partial Z_1}}{\, z=0}  & D[n_1, n_2, n_4 ]\restriction{\frac{\partial \theta_{n_3}}{\partial Z_1}}{\, z=0} \\   D[n_4, n_2, n_3] \restriction{\frac{\partial \theta_{n_1}}{\partial Z_2}}{\, z=0} & D[n_1, n_4, n_3]\restriction{\frac{\partial \theta_{n_2}}{\partial Z_2}}{\, z=0} & D[n_1, n_2, n_4 ]\restriction{\frac{\partial \theta_{n_3}}{\partial Z_2}}{\, z=0} \\  D[n_4, n_2, n_3]\restriction{\frac{\partial \theta_{n_1}}{\partial Z_3}}{\, z=0} & D[n_1, n_4, n_3]\restriction{\frac{\partial \theta_{n_2}}{\partial Z_3}}{\, z=0} &  D[n_1 , n_2, n_4]\restriction{\frac{\partial \theta_{n_3}}{\partial Z_3}}{\, z=0} \end{pmatrix}
\end{equation*}
\end{note}


\addcontentsline{toc}{chapter}{Bibliography}

\begin{thebibliography}{9}
\bibitem[ACGH85]{ACGH} E. Arbarello, M. Cornalba, P. A. Griffiths, J. Harris, {\it Geometry of algebraic
curves}, Die Grundlehren der matematischen Wissenschaften in Einzeldarstellungen 267, Springer-Verlag (1985);
\bibitem[CS03]{CapSer} L. Caporaso, E. Sernesi, {\it Recovering plane curves from their bitangents}, Journal of Algebraic Geometry 12 (2003), 225-244;
\bibitem[DFSM14]{F} F. Dalla Piazza, A. Fiorentino, R. Salvati Manni, {\it Plane quartics: the matrix of bitangents}, arXiv:1409.5032 (2014);
\bibitem[Do12]{Dolgachev} I.V. Dolgachev, {\it Classical algebraic geometry: a modern view}, Cambridge University Press (2012);
\bibitem[Fa79]{Fay} J. Fay, {\it On the Riemann-Jacobi formula}, Nachrichten der Akademie der Wissenschaften in G\"{o}ttingen (1979), 61-73;
\bibitem[Fr85]{Frobenius} G. F. Frobenius, {\it \"{U}ber die constanten Factoren der Thetareihen}, Journal f\"{u}r die reine und angewandte Mathematik 98 (1885), 244-263;
\bibitem[GH78]{GHprinciples} P. A. Griffiths, J. Harris, {\it Principles of algebraic geometry}, Wiley Classics Library, John Wiley \& Sons (1978);
\bibitem[GH04]{GH} B. H. Gross, J. Harris, {\it On some geometric constructions related to theta characteristics}, Contributions to automorphic forms, geometry and number theory, Johns Hopkins University Press, Baltimore
(2004), 279-311;
\bibitem[Gu02]{Guardia} J. Gu\`{a}rdia, {\it Jacobian Nullwerte and algebraic equations}, Journal of Algebra 253 (2002), 112-132;
\bibitem[Gu11]{Guardia2} J. Gu\`{a}rdia, {\it On the Torelli problem and Jacobian Nullwerthe in genus three}, The Michigan Mathematical Journal 60 n.1 (2011), 51-65;
\bibitem[Ig72]{Igusabook} J. Igusa, {\it Theta functions}, Springer-Verlag (1972);
\bibitem[Ig83]{Igusa} J. Igusa, {\it Multiplicity one theorem and problems related to Jacobi's formula}, American Journal of Mathematics 105 n. 1 (1983), 157-187;
\bibitem[Le05]{Lehavi} D. Lehavi, {\it Any smooth plane quartic can be reconstructed from its bitangents}, Israel
Journal of Mathemathics 146 (2005), 371-379;
\bibitem[NR15]{NR} E. Nart, C. Ritzenthaler, {\it A new proof of a Thomae-like Formula for non hyperelliptic genus 3 curves},  	arXiv:1503.01012 (2015);
\bibitem[RF74]{FR} H. E. Rauch, H. M. Farkas, {\it Theta functions with applications to Riemann surfaces}, The Williams and Wilkins Company, Baltimore (1974);
\bibitem[We76]{Weber} H. M. Weber, {\it Theorie der Abel’schen Functionen vom Geschlecht $3$}, Berlin: Druck und Verlag von Georg Reimer (1876);
\end{thebibliography}
\end{document}